\documentclass[a4paper,12pt]{article}
\usepackage{authblk}
\usepackage{amsmath}
\usepackage{amsfonts}
\usepackage{amssymb}
\usepackage{amsthm}
\usepackage{graphicx}
\newtheorem{theorem}{Theorem}

\newtheorem{remark}[theorem]{Remark}

\theoremstyle{remark}

\begin{document}
\title{Shells without drilling rotations: a representation theorem in the framework of the geometrically nonlinear 6--parameter resultant shell theory}
\author{Mircea B\^{\i}rsan%
\thanks{\,Mircea B\^{\i}rsan, Lehrstuhl f\"ur Nichtlineare Analysis und Modellierung, Fakult\"at f\"ur  Mathematik, Universit\"at Duisburg-Essen, Campus Essen, Thea-Leymann Str. 9, 45127 Essen, Germany, email: mircea.birsan@uni-due.de ; and
Department of Mathematics, University ``A.I. Cuza'' of Ia\c{s}i, 700506 Ia\c{s}i,  Romania}\,\,\,
and\, Patrizio Neff\,%
\thanks{Patrizio Neff, Head of Lehrstuhl f\"ur Nichtlineare Analysis und Modellierung, Fakult\"at f\"ur  Mathematik, Universit\"at Duisburg-Essen, Campus Essen, Thea-Leymann Str. 9, 45127 Essen, Germany, email: patrizio.neff@uni-due.de, Tel.: +49-201-183-4243}
}


\maketitle

\begin{abstract}
In the framework of the geometrically nonlinear 6--parameter resultant shell theory we give a characterization of the shells without drilling rotations. These are shells for which the strain energy function $W$ is invariant under the superposition of drilling rotations, i.e. $W$ is insensible to the arbitrary local rotations about the third director $\boldsymbol{d}_3\,$. For this type of shells we show that the strain energy density $W$ can be represented as a function of certain combinations of the shell deformation gradient $\boldsymbol{F}$  and the surface gradient of  $\boldsymbol{d}_3\,$, namely $W\big(\boldsymbol{F}^{ T}\boldsymbol{F} , \,
        \boldsymbol{F}^T \boldsymbol{d}_3 \,, \, \boldsymbol{F}^T\mathrm{Grad}_s\boldsymbol{d}_3 \big)$.
For the case of isotropic shells we present explicit forms of the strain energy function $W$ having this property.
\end{abstract}

\section{Introduction}\label{sec1}

There exist many variants of the shell theory in the literature and various different models for shells and plates \cite{Antman95,Ciarlet00}. Some of them are classical, such as the Kirchhoff--Love shell theory or the Reissner--Mindlin theory \cite{Love,Naghdi72,Sokolnikoff,Timoshenko51}. These approaches do not take into account the drilling rotations in shells, i.e. the rotations of points about the shell filament.
Although the classical theories can handle the majority of practical shell problems, there exist however situations when the drilling rotations cannot be neglected. For instance, in the case of shells with a certain microstructure or for branching and self-intersecting shells one should account also for the drilling rotations \cite{Pietraszkiewicz-book04,Pietraszkiewicz07,Libai98,Steigmann99,Steigmann13}. The question whether one has to consider a drilling degree of freedom or not is intensively discussed in the literature.  Since in the classical
Kirchhoff--Love and  Reissner--Mindlin there is no drilling degree of freedom present, there is also no energy storage due to this variable.

Other non-classical approaches are more general, such as the Cosserat--type theories of shells or micropolar shells \cite{Cosserat09neu,Altenbach-Erem-Review}. A general theory of micropolar shells has been developed by Eremeyev and Zubov \cite{Eremeyev08}. It is conceivable, that a Cosserat--type shell model has the possibility of drilling degrees of freedom. In this case it is a matter of assumption of what is the contribution to the  stored elastic energy. For these generalized models each material point has six degrees of freedom (parameters): 3 for the translations and 3 rotational degrees of freedom. In order to characterize  the independent rotations of material points, one considers a triad of vectors  $\,\{\boldsymbol{d}_i\}$ (called \emph{directors}) attached to every point. The drilling rotations are then described as rotations about the third director $\boldsymbol{d}_3\,$.

It is difficult to physically justify the drilling degree of freedom. Therefore we wish to characterize the shells \emph{without} drilling rotations, i.e. the case when the strain energy of the shell is insensible to rotations about  the   director $\boldsymbol{d}_3\,$. In the framework of the 6-parameter resultant shell theory we prove a representation theorem for shells without drilling rotations, which asserts that the strain energy density $W$ of such shells can be represented as a function of the following arguments
$$W = W\big(\boldsymbol{F}^{ T}\boldsymbol{F} , \,
        \boldsymbol{F}^T \boldsymbol{d}_3 \,, \, \boldsymbol{F}^T\mathrm{Grad}_s\boldsymbol{d}_3 \big),  $$
where $\boldsymbol{F}$ is the shell deformation gradient tensor and $\,\mathrm{Grad}_s\,$ denotes the surface gradient operator.

It is now firmly established  that the  6-parameter resultant shell theory \cite{Libai98,Pietraszkiewicz-book04,Pietraszkiewicz04} has the same kinematical structure as the theory of Cosserat shells. The shells of Cosserat--type without drilling rotations have been studied extensively by Zhilin in \cite{Zhilin76,Zhilin06}. We compare next our results with the model of Zhilin and find a close relation, especially in the linearized theory.

In the last section we consider the explicit form of the strain energy function for isotropic shells made of a physically linear material. The existence of minimizers for 6--parameter elastic shells has been proved in \cite{Birsan-Neff-JElast-2013,Birsan-Neff-MMS-2013} under certain conditions on the constitutive coefficients which insure the positive definiteness of $W$.
However, the existence theorem presented in  \cite{Birsan-Neff-MMS-2013} does not apply to shells without drilling rotations, since we obtain here a limit case when the strain energy function is only positive semi-definite. In this case one could adapt the methods presented for Cosserat plates in \cite{Neff_plate04_cmt,Neff_plate07_m3as} and prove the existence of minimizers for shells without drilling rotations.

This discussion on the drilling degree of freedom is intimately related to the value of the Cosserat couple modulus $\,\mu_c\,$ in the shell models derived by dimensional reduction from three-dimensional  Cosserat elasticity, where the parent model is a hyperelastic micropolar model \cite{Neff_zamm06,Neff_Cosserat_plasticity05,Neff_Danzig05}. Thus,  the case of shells without drilling degree of freedom corresponds to $\,\mu_c=0\,$. This is a degenerate case, which was studied extensively by Neff in \cite{Neff_plate07_m3as}, see also \cite{Neff_Jeong_Conformal_ZAMM08,Jeong_Neff_ZAMM08,Neff_Munch_CMT09,Jeong_Neff_MMS2010}. In general, it can be said, that the nonlinear theories without drill are \emph{not} well-posed as minimization problems, since a certain control of the third director in combination with the coupling of the director to the shell surface is missing. Only in the linearized models this problem can be solved \cite{Neff_Habil04,Neff_Danzig05,Neff_Chelminski_ifb07,Neff_Hong_Reissner08}. Thus, from a mathematical point of view, the inclusion of the drilling degree of freedom stabilizes somehow the model at the expense of physical clarity.

\section{Field equations for 6--parameter elastic shells}\label{sec2}
\vspace{1pt}

In this section we repeat the basic equations of the 6--parameter resultant shell theory and  present some useful relations.

Let us denote with $S^0$ the base surface of a general shell in the reference configuration and with $S$ the  base surface   in the deformed configuration. The shell is referred to a fixed Cartesian frame with origin $O$ and unit vectors $\{\boldsymbol{e}_1 , \boldsymbol{e}_2 , \boldsymbol{e}_3\}$ along the coordinate axes. The reference configuration of the shell will be described by the position vector $\boldsymbol{y}^0$  and the structure tensor $\boldsymbol{Q}^{0}$, which is a proper orthogonal tensor. In order to represent the structure tensor $\boldsymbol{Q}^{0}$ we consider an orthonormal triad of vectors $\{\boldsymbol{d}^0_1 , \boldsymbol{d}^0_2 , \boldsymbol{d}^0_3\}$ (called \emph{directors}) attached to every point of $S^0$ such that $\boldsymbol{Q}^{0} = \boldsymbol{d}_i^0  \otimes \boldsymbol{e}_i\,$ \cite{Pietraszkiewicz-book04,Eremeyev06}. Here we employ the summation convention over repeated indexes. The Latin indexes $i,j,...$ take the values $\{1,2,3\}$, while Greek indexes $\alpha,\beta,...$ the values $\{1,2\}$.

We denote the material curvilinear coordinates on the surface $S^0$ with $(x_1,x_2)\in \omega$. The set $\omega$ is assumed to be a bounded open domain with Lipschitz boundary $\partial\omega$ in the $Ox_1x_2$ plane. Then the reference configuration  of the shell is characterized by the functions
\begin{equation*}
\begin{array}{l}
    \,\boldsymbol{y}^0:\omega\subset \mathbb{R}^2\rightarrow\mathbb{R}^3,\qquad\qquad\,\,\boldsymbol{y}^0=\boldsymbol{y}^0(x_1,x_2) ,\\
    \boldsymbol{Q}^{0}:\omega\subset \mathbb{R}^2\rightarrow \mathrm{SO}(3), \qquad\,\,\, \boldsymbol{Q}^{0} = \boldsymbol{d}_i^0 (x_1,x_2)\otimes \boldsymbol{e}_i\,.
    \end{array}
\end{equation*}

Let us designate by $\boldsymbol{y}(x_1,x_2)$ the position vector of the points of $S$ and by $ \boldsymbol{R}(x_1,x_2) = \boldsymbol{d}_i (x_1,x_2)\otimes \boldsymbol{e}_i\in \mathrm{SO}(3)\,$ the structure tensor in the deformed configuration. Here $\{\boldsymbol{d}_i(x_1,x_2)\}$ is the orthonormal triad of directors attached to the point of $S$ with the initial curvilinear coordinates $(x_1,x_2)$. Then, one can characterize the deformation of the elastic shell by means of the functions
\begin{equation*}
    \boldsymbol{y}=\boldsymbol{\chi}(\boldsymbol{y}^0),\qquad \boldsymbol{Q}^e= \boldsymbol{d}_i\otimes \boldsymbol{d}_i^0\in \mathrm{SO}(3),
\end{equation*}
where $\boldsymbol{\chi}:S^0\rightarrow\mathbb{R}^3$ represents the deformation of the base surface of the shell and $\boldsymbol{Q}^e=\boldsymbol{R}\,\boldsymbol{Q}^{0,T} $ describes the (effective)  elastic rotation.

For the reference base surface $S^0$ we denote with $\{\boldsymbol{a}_1,\boldsymbol{a}_2 \}$ the (covariant) base vectors given by $\boldsymbol{a}_\alpha=   { \partial\boldsymbol{y}^0}/{\partial x_\alpha}\,$ and with $\boldsymbol{n}^0 = { \boldsymbol{a}_1\times\boldsymbol{a}_2}/{ \| \boldsymbol{a}_1\times\boldsymbol{a}_2\|}\,$ the unit normal to $S^0$. We introduce also the reciprocal (contravariant) base vectors $\{\boldsymbol{a}^1,\boldsymbol{a}^2 \}$ and the notations $\boldsymbol{a}_3=\boldsymbol{a}^3= \boldsymbol{n}^0$ such that the relations $\boldsymbol{a}_i\cdot \boldsymbol{a}^j=\delta_i^j$ (the Kronecker symbol) holds. In the reference configuration $S^0$ we choose the initial director $\boldsymbol{d}^0_3$ such that $\boldsymbol{d}^0_3=\boldsymbol{n}^0$.

In order to present the elastic shell strain measures and bending--curvature measures we designate by $\partial_\alpha  \equiv \tfrac{\partial }{\partial x_\alpha}\,$ the partial derivative with respect to $x_\alpha$ and by $\boldsymbol{F}=\text{Grad}_s\boldsymbol{y}=\partial_\alpha \boldsymbol{y}\otimes\boldsymbol{a}^\alpha$ the shell deformation gradient tensor. We also employ the notations
\begin{equation}\label{0,5}
\begin{array}{l}
    \boldsymbol{a}=\boldsymbol{a}_\alpha\otimes \boldsymbol{a}^\alpha=
    a_{\alpha\beta}\boldsymbol{a}^\alpha\otimes \boldsymbol{a}^\beta= a^{\alpha\beta}\boldsymbol{a}_\alpha\otimes \boldsymbol{a}_\beta=
    \boldsymbol{d}^0_\alpha\otimes \boldsymbol{d}^0_\alpha,\quad
    a=\sqrt{\det(a_{\alpha\beta})_{2\times 2}}>0,\vspace{4pt}\\
     \boldsymbol{b}= -\text{Grad}_s\boldsymbol{n}^0=-\partial_\alpha \boldsymbol{n}^0\otimes\boldsymbol{a}^\alpha,\\
     \boldsymbol{c}=-\boldsymbol{n}^0\times\boldsymbol{a} =  -\boldsymbol{a}\times  \boldsymbol{n}^0 = \dfrac{1}{\sqrt{a}}\,\epsilon_{\alpha\beta}\,\boldsymbol{a}_\alpha\otimes \boldsymbol{a}_\beta = \sqrt{a}\,\epsilon_{\alpha\beta}\,\boldsymbol{a}^\alpha\otimes \boldsymbol{a}^\beta = \epsilon_{\alpha\beta}\,\boldsymbol{d}^0_\alpha\otimes \boldsymbol{d}^0_\beta\,,
     \end{array}
\end{equation}
where $\epsilon_{\alpha\beta}\,$ is the two-dimensional alternator ($\epsilon_{12}=-\epsilon_{21}=1\,,\,\epsilon_{11}=\epsilon_{22}=0$). The tensors $\boldsymbol{a}$ and $\boldsymbol{b}$ are the first and second fundamental tensors of the surface $S^0$, while $\boldsymbol{c}$ is called the alternator tensor of $S^0$ \cite{Zhilin06}. We have $\,\,\boldsymbol{a}^T=\boldsymbol{a}$ , $\boldsymbol{b}^T=\boldsymbol{b}\,\,$ and $\,\,\boldsymbol{c}^T=-\boldsymbol{c}$.

Then, the  elastic shell strain tensor $\boldsymbol{E}^e$ in the material representation is \cite{Eremeyev06,Pietraszkiewicz-book04}
\begin{equation}\label{1}
    \boldsymbol{E}^e=\boldsymbol{Q}^{e,T}\text{Grad}_s\,\boldsymbol{y}- \text{Grad}_s\,\boldsymbol{y}^0=\big(\boldsymbol{Q}^{e,T}\partial_\alpha \boldsymbol{y}- \partial_\alpha \boldsymbol{y}^0 \big) \otimes \boldsymbol{a}^\alpha\,,
\end{equation}

The bending--curvature  tensor $\boldsymbol{K}^e$ in  material representation is given by \cite{Eremeyev06,Pietraszkiewicz-book04,Birsan-Neff-JElast-2013}
\begin{equation}\label{2}
    \boldsymbol{K}^e= \boldsymbol{Q}^{e,T}\text{axl}(\partial_\alpha \boldsymbol{Q}^e  \boldsymbol{Q}^{e,T}) \otimes \boldsymbol{a}^\alpha= \text{axl}(\boldsymbol{Q}^{e,T} \partial_\alpha \boldsymbol{Q}^e ) \otimes \boldsymbol{a}^\alpha ,
\end{equation}
We can represent $\boldsymbol{K}^e$ in terms of $\boldsymbol{R}$ and $\boldsymbol{Q}^0$ in the form
\begin{equation}\label{3}
    \boldsymbol{K}^e= \boldsymbol{Q}^{0} \big[ \text{axl}(\boldsymbol{R}^T \partial_\alpha \boldsymbol{R}  )- \text{axl}(\boldsymbol{Q}^{0,T}\partial_\alpha \boldsymbol{Q}^{0})\big] \otimes \boldsymbol{a}^\alpha\,,
\end{equation}
or equivalently
\begin{equation}\label{4}
\begin{array}{c}
    \boldsymbol{K}^e=\boldsymbol{K} - \boldsymbol{K}^0,\qquad \text{with} \qquad \boldsymbol{K}= \boldsymbol{Q}^{0}   \text{axl}(\boldsymbol{R}^T \partial_\alpha \boldsymbol{R}  ) \otimes \boldsymbol{a}^\alpha\,,\vspace{4pt} \\
      \boldsymbol{K}^0= \boldsymbol{Q}^{0}  \text{axl}(\boldsymbol{Q}^{0,T}\partial_\alpha \boldsymbol{Q}^{0})  \otimes \boldsymbol{a}^\alpha\,\,=\,\,\text{axl}(\partial_\alpha \boldsymbol{Q}^{0}\,\boldsymbol{Q}^{0,T}) \otimes \boldsymbol{a}^\alpha\,.
    \end{array}
\end{equation}
It is useful to express the tensors $\boldsymbol{E}^e$ and $\boldsymbol{K}^e$ decomposed in the tensor basis $\{\boldsymbol{d}_i^0\otimes\boldsymbol{a}^\alpha\}$. Thus, we obtain
\begin{equation}\label{5}
    \begin{array}{l}
    \boldsymbol{E}^e=[(\partial_\alpha \boldsymbol{y}\cdot\boldsymbol{d}_i)\boldsymbol{d}_i^0 -\boldsymbol{a}_\alpha]\otimes\boldsymbol{a}^\alpha,\vspace{4pt}\\
    \boldsymbol{K}=[(\partial_\alpha \boldsymbol{d}_2\cdot\boldsymbol{d}_3)\boldsymbol{d}_1^0
    +(\partial_\alpha \boldsymbol{d}_3\cdot\boldsymbol{d}_1)\boldsymbol{d}_2^0
    +(\partial_\alpha \boldsymbol{d}_1\cdot\boldsymbol{d}_2)\boldsymbol{d}_3^0\,] \otimes\boldsymbol{a}^\alpha,\vspace{4pt}\\
     \boldsymbol{K}^0=[(\partial_\alpha \boldsymbol{d}_2^0\cdot\boldsymbol{d}_3^0)\boldsymbol{d}_1^0
    +(\partial_\alpha \boldsymbol{d}_3^0\cdot\boldsymbol{d}_1^0)\boldsymbol{d}_2^0
    +(\partial_\alpha \boldsymbol{d}_1^0\cdot\boldsymbol{d}_2^0)\boldsymbol{d}_3^0\,] \otimes\boldsymbol{a}^\alpha.
    \end{array}
\end{equation}
Thus, the tensor components are
\begin{equation}\label{5,5}
    \begin{array}{l}
    \boldsymbol{E}^e=E_{i\alpha}^e\boldsymbol{d}_i^0 \otimes\boldsymbol{a}^\alpha= (\partial_\alpha \boldsymbol{y}\cdot\boldsymbol{d}_i -\boldsymbol{a}_\alpha\cdot\boldsymbol{d}_i^0)\,\boldsymbol{d}_i^0 \otimes\boldsymbol{a}^\alpha, \vspace{4pt}\\
       \boldsymbol{K}^e=K_{i\alpha}^e\boldsymbol{d}_i^0 \otimes\boldsymbol{a}^\alpha=
        \frac{1}{2}\,e_{ijk}(\partial_\alpha \boldsymbol{d}_j\cdot\boldsymbol{d}_k- \partial_\alpha \boldsymbol{d}_j^0\cdot\boldsymbol{d}_k^0)\,\boldsymbol{d}_i^0 \otimes\boldsymbol{a}^\alpha\vspace{4pt}\\
       \quad\,\,\,\,\,
       = (\partial_\alpha \boldsymbol{d}_2\cdot\boldsymbol{d}_3- \partial_\alpha \boldsymbol{d}_2^0\cdot\boldsymbol{d}_3^0)\,\boldsymbol{d}_1^0 \otimes\boldsymbol{a}^\alpha
       + (\partial_\alpha \boldsymbol{d}_3\cdot\boldsymbol{d}_1- \partial_\alpha \boldsymbol{d}_3^0\cdot\boldsymbol{d}_1^0)\,\boldsymbol{d}_2^0 \otimes\boldsymbol{a}^\alpha
       \vspace{4pt}\\
       \qquad\quad+(\partial_\alpha \boldsymbol{d}_1\cdot\boldsymbol{d}_2- \partial_\alpha \boldsymbol{d}_1^0\cdot\boldsymbol{d}_2^0)\,\boldsymbol{d}_3^0 \otimes\boldsymbol{a}^\alpha,
    \end{array}
\end{equation}
where $e_{ijk}$ is the three-dimensional alternator (i.e. the signature of the permutation $(1,2,3)\,\rightarrow\,(i,j,k)$).

Let the tensors $\boldsymbol{N}$ and $\boldsymbol{M}$ be  the internal surface stress resultant and stress couple tensors of the 1$^{st}$ Piola--Kirchhoff type for the shell.
The equilibrium equations for 6--parameter shells are \cite{Pietraszkiewicz-book04,Eremeyev06,Birsan-Neff-MMS-2013}
\begin{equation}\label{6}
\begin{array}{l}
    \mathrm{Div}_s\, \boldsymbol{N}+\boldsymbol{f}=\boldsymbol{0},\vspace{4pt}\\
    \mathrm{Div}_s\, \boldsymbol{M} + \mathrm{axl}(\boldsymbol{N}\boldsymbol{F}^T-\boldsymbol{F}\boldsymbol{N}^T)
    +\boldsymbol{l}=\boldsymbol{0},
    \end{array}
\end{equation}
where $\boldsymbol{f}$ and $\boldsymbol{l}$  are the external surface resultant force and couple vectors applied to points of $S$, but measured per unit area of $S^0\,$. Here
$\mathrm{Div}_s$ denotes the surface divergence  and  $\,\mathrm{axl}(\,\cdot\,)$ represents the axial vector of a skew--symmetric tensor. To formulate the boundary--value problem, we consider boundary conditions of the type \cite{Pietraszkiewicz11}
\begin{equation}\label{7}
\begin{array}{l}
\boldsymbol{N}\boldsymbol{\nu}=\boldsymbol{n}^*,\qquad \boldsymbol{M}\boldsymbol{\nu}=\boldsymbol{m}^*\qquad\mathrm{along}\,\,\,\partial S^0_f\,, \vspace{4pt}\\
    \quad\boldsymbol{y}=\boldsymbol{y}^* ,\qquad\quad\,\, \boldsymbol{R}=\boldsymbol{R}^* \qquad\mathrm{along}\,\,\,\partial S^0_d\,,
    \end{array}
\end{equation}
where $\boldsymbol{\nu}$ is the external unit normal vector to the boundary curve $\partial S^0$ (lying in the tangent plane) and $\{ \partial S^0_f\,, \partial S^0_d\,\}$ is a disjoint partition of $\partial S^0\,$.

Let $W=W(\boldsymbol{E}^e,\boldsymbol{K}^e)$ be  the strain energy density of the elastic shell,  measured per unit area of the base surface $S^0\,$. The principle of energy can be written in the form \cite{Zhilin06}
\begin{equation}\label{8}
    \dot W = (\boldsymbol{Q}^{e,T}\boldsymbol{N})\cdot\dot{\boldsymbol{E}^e} + (\boldsymbol{Q}^{e,T}\boldsymbol{M})\cdot\dot{\boldsymbol{K}^e},
\end{equation}
where a superposed dot designates the material time derivative and $\,\cdot\,$ means the scalar product of tensors, i.e. $\boldsymbol{S}\cdot\boldsymbol{V}=\mathrm{tr}(\boldsymbol{S}^T\boldsymbol{V})$.
Similar relations expressing the internal virtual power and the principle of virtual work have been presented in \cite{Pietraszkiewicz04,Eremeyev06}. Under the hyperelasticity assumption, $\boldsymbol{N}$ and $\boldsymbol{M}$ are expressed by the constitutive equations
\begin{equation}\label{9}
    \boldsymbol{N}=\boldsymbol{Q}^e\,\dfrac{\partial\, W}{\partial \boldsymbol{E}^e}\,\,,\qquad \boldsymbol{M}=\boldsymbol{Q}^e\,\dfrac{\partial\, W}{\partial \boldsymbol{K}^e}\,\,.
\end{equation}

To resume, the  boundary--value problem for non-linear elastic 6--parameter shells consists of the equations \eqref{1}, \eqref{2}, \eqref{6}, \eqref{7}, and \eqref{9}. The minimization problem associated to the deformation of elastic shells can be put in the following form: find the pair $(\hat{\boldsymbol{y}},\hat{\boldsymbol{R}})$ in the admissible set $\mathcal{A}$ which realizes the minimum of the functional
\begin{equation*}
I(\boldsymbol{y},\boldsymbol{R})=\int_{S^0} W(\boldsymbol{E}^e,\boldsymbol{K}^e)\,\mathrm{d}S - \Lambda(\boldsymbol{y},\boldsymbol{R})\qquad\mathrm{for}\qquad (\boldsymbol{y},\boldsymbol{R})\in \mathcal{A},
\end{equation*}
where d$S$ is the area element of the surface $S^0\, $ and  $\Lambda(\boldsymbol{y},\boldsymbol{R})$ is a function representing the potential of external surface loads $\boldsymbol{f}$, $\boldsymbol{l}$, and boundary loads $\boldsymbol{n}^*$, $\boldsymbol{m}^*$  \cite{Pietraszkiewicz04,Birsan-Neff-MMS-2013}.
The admissible set   is
\begin{equation*}
    \mathcal{A}=\big\{(\boldsymbol{y},\boldsymbol{R})\in\boldsymbol{H}^1(\omega, \mathbb{R}^3)\times\boldsymbol{H}^1(\omega, \mathrm{SO}(3))\,\,\big|\,\,\,  \boldsymbol{y}_{\big| \partial S^0_d}=\boldsymbol{y}^*, \,\,\boldsymbol{R}_{\big| \partial S^0_d}=\boldsymbol{R}^* \big\},
\end{equation*}
where $\boldsymbol{H}^1$ designates the well-known Sobolev space and the boundary conditions are to be understood in the sense of traces.

We remark that the 6--parameter model of shells takes into account also the deformations caused by the so-called \emph{drilling rotations}. The drilling rotation in a point of $S$ is interpreted as the rotation about the director $\boldsymbol{d}_3\,$.

In what follows, we  describe the shells which are insensible to drilling rotations, in the framework of the 6--parameter theory. This class of shells is important since it is widely used in applications.

\section{Characterization of shells without drilling rotations}\label{sec3}

For shells completely without drilling rotations the strain energy $W$ must be insensible  to the rotations  about $\boldsymbol{d}_3 \,$. Let $\boldsymbol{R}_\theta $ denote the rotation of angle $\theta(x_1,x_2)$ about $\boldsymbol{d}_3 \,$. The general form of $\boldsymbol{R}_\theta $ may be given by
\begin{equation*}
    \begin{array}{l}
    \boldsymbol{R}_\theta=\boldsymbol{d}_3\otimes\boldsymbol{d}_3+ \cos \theta(\boldsymbol{1}-\boldsymbol{d}_3\otimes\boldsymbol{d}_3)+\sin\theta (\boldsymbol{d}_3\times \boldsymbol{1}),
\end{array}
\end{equation*}
or equivalently, written in the $\{\boldsymbol{d}_i\otimes\boldsymbol{d}_j\}$ tensor basis,
\begin{equation}\label{10}
    \begin{array}{c}
      \boldsymbol{R}_\theta= \cos \theta( \boldsymbol{d}_1\otimes\boldsymbol{d}_1+ \boldsymbol{d}_2\otimes\boldsymbol{d}_2 )+\sin\theta (\boldsymbol{d}_2\otimes\boldsymbol{d}_1- \boldsymbol{d}_1\otimes\boldsymbol{d}_2) + \boldsymbol{d}_3\otimes\boldsymbol{d}_3\,.
    \end{array}
\end{equation}
In other words, the shells without drilling rotations are characterized by the property that $ W( {\boldsymbol{E}^e}, {\boldsymbol{K}^e})$   remains invariant under the superposition of arbitrary rotations angle $\theta(x_1,x_2)$ about $\boldsymbol{d}_3 \,$, i.e. invariant under the transformation
\begin{equation}\label{11}
    \begin{array}{c}
      \boldsymbol{Q}^e\,\,\,\longrightarrow\,\,\, \boldsymbol{R}_\theta \boldsymbol{Q}^e\,\,.
\end{array}
\end{equation}
Taking into account the relations \eqref{1}, \eqref{2}, and \eqref{11}, we see that the strain and bending-curvature measures ${\boldsymbol{E}^e}$ and $ {\boldsymbol{K}^e}$ transform as
\begin{equation}\label{12}
    \begin{array}{l}
      \boldsymbol{E}^e\,\,\,\longrightarrow\,\,\, \boldsymbol{Q}^{e,T}\boldsymbol{R}_\theta^T\boldsymbol{F} - \boldsymbol{a} ,
      \vspace{4pt}\\
      \boldsymbol{K}^e \,\,\,\longrightarrow\,\,\,  \text{axl}[\boldsymbol{Q}^{e,T} \boldsymbol{R}_\theta^T \partial_\alpha (\boldsymbol{R}_\theta\boldsymbol{Q}^e )] \otimes \boldsymbol{a}^\alpha
    \end{array}
\end{equation}
In view of \eqref{1}, \eqref{2}, and \eqref{12}, the invariance condition on the strain function $W$ can be written as
\begin{equation}\label{13}
    \begin{array}{l}
       W( \boldsymbol{Q}^{e,T}\boldsymbol{F} - \boldsymbol{a} \,\,,\,\, \text{axl}(\boldsymbol{Q}^{e,T}  \partial_\alpha \boldsymbol{Q}^e ) \otimes \boldsymbol{a}^\alpha) \vspace{4pt}\\
       \qquad\qquad=
      W\big(\boldsymbol{Q}^{e,T}\boldsymbol{R}_\theta^T\boldsymbol{F} - \boldsymbol{a} \,\,,\,\,
       \text{axl}[\boldsymbol{Q}^{e,T} \boldsymbol{R}_\theta^T \partial_\alpha (\boldsymbol{R}_\theta\boldsymbol{Q}^e )] \otimes \boldsymbol{a}^\alpha\big)
    \end{array}
\end{equation}
for all rotation angles $\theta(x_1,x_2)$.

If we employ the tensors components $E_{i\alpha}^e$ and $K_{i\alpha}^e$ introduced in \eqref{5,5}, then the invariance condition \eqref{13} can be put into another form: we denote by $\boldsymbol{d}_i^\theta$ the directors rotated with angle $\theta$ about $\boldsymbol{d}_3$ in the configuration, i.e.
\begin{equation}\label{14}
    \begin{array}{c}
      \boldsymbol{d}_1^\theta= \boldsymbol{R}_\theta\boldsymbol{d}_1=\cos\theta\,\boldsymbol{d}_1+ \sin\theta\,\boldsymbol{d}_2\,,\qquad
    \boldsymbol{d}_2^\theta= \boldsymbol{R}_\theta\boldsymbol{d}_2=-\sin\theta\,\boldsymbol{d}_1+ \cos\theta\,\boldsymbol{d}_2\,,\vspace{4pt}\\
    \boldsymbol{d}_3^\theta= \boldsymbol{R}_\theta\boldsymbol{d}_3=\boldsymbol{d}_3\,,\qquad
    \boldsymbol{R}_\theta\boldsymbol{Q}^e= \boldsymbol{d}_i^\theta\otimes \boldsymbol{d}_i^0\,.
    \end{array}
\end{equation}
Then, the transformation \eqref{12} can be written with the help of the components \eqref{5,5} as follows
\begin{equation}\label{15}
    \begin{array}{l}
      E_{i\alpha}^e\,\boldsymbol{d}_i^0 \otimes\boldsymbol{a}^\alpha\,\,\,\longrightarrow\,\,\, (\partial_\alpha \boldsymbol{y}\cdot\boldsymbol{d}_i^\theta -\partial_\alpha \boldsymbol{y}^0 \cdot\boldsymbol{d}_i^0)\,\boldsymbol{d}_i^0 \otimes\boldsymbol{a}^\alpha ,
      \vspace{4pt}\\
      K_{i\alpha}^e\,\boldsymbol{d}_i^0 \otimes\boldsymbol{a}^\alpha \,\,\,\longrightarrow\,\,\,
      \frac{1}{2}\,e_{ijk}(\partial_\alpha \boldsymbol{d}_j^\theta\cdot \boldsymbol{d}_k^\theta- \partial_\alpha \boldsymbol{d}_j^0\cdot\boldsymbol{d}_k^0)\,\boldsymbol{d}_i^0 \otimes\boldsymbol{a}^\alpha
    \end{array}
\end{equation}
and the invariance condition \eqref{13} on the strain energy function $W$ takes the form
\begin{equation}\label{16}
    \begin{array}{c}
      \tilde W(E_{i\alpha}^e\,,\,K_{i\alpha}^e\,)= \tilde W\big( \partial_\alpha \boldsymbol{y}\cdot\boldsymbol{d}_i^\theta -\partial_\alpha \boldsymbol{y}^0 \cdot\boldsymbol{d}_i^0\,,\, \frac{1}{2}\,e_{ijk}(\partial_\alpha \boldsymbol{d}_j^\theta\cdot \boldsymbol{d}_k^\theta- \partial_\alpha \boldsymbol{d}_j^0\cdot\boldsymbol{d}_k^0)\,\big),
    \end{array}
\end{equation}
where the components $E_{i\alpha}^e$ and $K_{i\alpha}^e$  are given in \eqref{5,5} and the function $\tilde W$ is defined by
$\tilde W(E_{i\alpha}^e\,,\,K_{i\alpha}^e)= W(E_{i\alpha}^e\boldsymbol{d}_i^0\otimes \boldsymbol{a}^\alpha\,,\,K_{i\alpha}^e\boldsymbol{d}_i^0\otimes \boldsymbol{a}^\alpha)$.

We present now the main result, which gives a representation theorem for the strain energy function $W$ that satisfy the invariance condition \eqref{13}, i.e. for shells without drilling rotations.

\begin{theorem}
The strain energy function $W$ of the shell is invariant under the transformations \eqref{11} (i.e. it is insensible to drilling rotations) if and only if it can be represented as a function of the following arguments
\begin{equation}\label{17}
    \begin{array}{c}
      W =
      \hat W\big(\boldsymbol{F}^{ T}\boldsymbol{F} \,\,, \,\,
        \boldsymbol{F}^T \boldsymbol{d}_3 \,\,, \,\,  \boldsymbol{F}^T\mathrm{Grad}_s\boldsymbol{d}_3 \big).
\end{array}
\end{equation}
\end{theorem}
\begin{proof}
We observe first that any function $W$ which admits a representation of the type \eqref{17} is invariant under the transformations \eqref{11}. Indeed, the vectors $\boldsymbol{y}$ and $\boldsymbol{d}_3$ are invariant under rotations about $\boldsymbol{d}_3\,$. Therefore, the tensors $\boldsymbol{F}=\mathrm{Grad}_s\boldsymbol{y}$ and $\mathrm{Grad}_s\boldsymbol{d}_3$ are also invariant under rotations \eqref{10} and consequently, any function $\hat W$ of the combinations $(\boldsymbol{F}^{ T}\boldsymbol{F}  \,, \,         \boldsymbol{F}^T \boldsymbol{d}_3 \, ,  \,  \boldsymbol{F}^T\mathrm{Grad}_s\boldsymbol{d}_3 \big)$ is invariant under the transformation \eqref{11}.

Conversely, let us prove that any strain energy function $W$ satisfying this invariance condition admits necessarily a representation of the form \eqref{17}. We assume that $W$ fulfills the invariance relation \eqref{13}, which is equivalent to \eqref{16}. In order to write \eqref{16} in a more convenient form, we calculate the scalar products
\begin{equation*}
    \begin{array}{l}
      \partial_\alpha \boldsymbol{d}_2^\theta\cdot \boldsymbol{d}_3^\theta = \boldsymbol{d}_3 \cdot\partial_\alpha \boldsymbol{d}_2^\theta\,,\vspace{3pt}\\
      \partial_\alpha \boldsymbol{d}_3^\theta\cdot \boldsymbol{d}_1^\theta =- \boldsymbol{d}_3^\theta\cdot
      \partial_\alpha \boldsymbol{d}_1^\theta  = - \boldsymbol{d}_3\cdot
      \partial_\alpha \boldsymbol{d}_1^\theta\,,\vspace{3pt}\\
       \partial_\alpha \boldsymbol{d}_1^\theta\cdot \boldsymbol{d}_2^\theta =  \big[
       (\cos\theta\,\partial_\alpha\boldsymbol{d}_1+ \sin\theta\,\partial_\alpha\boldsymbol{d}_2)+(-\sin\theta\,\boldsymbol{d}_1+ \cos\theta\,\boldsymbol{d}_2)\partial_\alpha\theta\big] \!\cdot \! (-\sin\theta\,\boldsymbol{d}_1+ \cos\theta\,\boldsymbol{d}_2)\vspace{3pt}\\
       \qquad\qquad\,= \partial_\alpha \boldsymbol{d}_1\cdot \boldsymbol{d}_2 + \partial_\alpha \theta\,.
    \end{array}
\end{equation*}
Using relations of this type we can put the condition \eqref{16} in the following form
\begin{equation}\label{18}
    \begin{array}{r}
       \tilde W(E_{i\alpha}^e\,,\,K_{1\alpha}^e\,,\,K_{2\alpha}^e\,,\,K_{3\alpha}^e\,)= \tilde W\big( \partial_\alpha \boldsymbol{y}\cdot\boldsymbol{d}_i^\theta -\boldsymbol{a}_\alpha \cdot\boldsymbol{d}_i^0\,,\,
       \partial_\alpha \boldsymbol{d}_2^\theta\cdot \boldsymbol{d}_3- \partial_\alpha \boldsymbol{d}_2^0\cdot\boldsymbol{n}^0\,,\,
         \vspace{3pt}\\
         \qquad\qquad\,-\partial_\alpha \boldsymbol{d}_1^\theta\cdot \boldsymbol{d}_3+ \partial_\alpha \boldsymbol{d}_1^0\cdot\boldsymbol{n}^0\,,\, \partial_\alpha \theta+ \partial_\alpha \boldsymbol{d}_1\cdot \boldsymbol{d}_2 -  \partial_\alpha \boldsymbol{d}_1^0\cdot \boldsymbol{d}_2^0\,\big),
    \end{array}
\end{equation}
where $\boldsymbol{d}_i^\theta$ are expressed by \eqref{14}. Relation \eqref{18} must hold for all angles of rotation $\theta=\theta(x_1,x_2)$. Since the left-hand side is independent of $\theta$ and $\partial_\alpha\theta$, it follows that the derivatives of the right-hand side with respect to $\theta$ and $\partial_\alpha\theta$ are zero.
Thus, taking the derivative of \eqref{18} with respect to  $\partial_\alpha\theta$ we obtain by the chain rule
\begin{equation}\label{19}
    \dfrac{\partial\tilde W}{\partial \,K_{3\alpha}^e}\,= {0}, \qquad\mathrm{or\,\,equivalently}\qquad
    \dfrac{\partial\tilde W}{\partial(\boldsymbol{n}^0 \boldsymbol{K}^e)}\,=\boldsymbol{0},
\end{equation}
since $\boldsymbol{n}^0\boldsymbol{K}^e=\boldsymbol{d}_3^0(K_{i\alpha}^e\,\boldsymbol{d}_i^0\otimes \boldsymbol{a}^\alpha)=K_{3\alpha}\boldsymbol{a}^\alpha$. If we differentiate the relation \eqref{18} with respect to $\theta$ and use the relations
\begin{equation*}
    \begin{array}{c}
      \dfrac{\mathrm{d}}{\mathrm{d}\theta}(\boldsymbol{d}_1^\theta)=\boldsymbol{d}_2^\theta\,, \qquad \dfrac{\mathrm{d}}{\mathrm{d}\theta}(\boldsymbol{d}_2^\theta)=-\boldsymbol{d}_1^\theta\,, \vspace{3pt}\\
      \dfrac{\mathrm{d}}{\mathrm{d}\theta}(\partial_\alpha \boldsymbol{d}_2^\theta\cdot \boldsymbol{d}_3)=- \partial_\alpha \boldsymbol{d}_1^\theta\cdot \boldsymbol{d}_3\,, \qquad \dfrac{\mathrm{d}}{\mathrm{d}\theta}(\partial_\alpha \boldsymbol{d}_1^\theta\cdot \boldsymbol{d}_3)=\partial_\alpha \boldsymbol{d}_2^\theta\cdot \boldsymbol{d}_3 \,,
    \end{array}
\end{equation*}
then we get
\begin{equation}\label{20}
    \begin{array}{l}
      \dfrac{\partial\tilde W}{\partial \,E_{1\alpha}^e}\,\cdot(\partial_\alpha \boldsymbol{y}\cdot\boldsymbol{d}_2^\theta)+  \dfrac{\partial\tilde W}{\partial \,E_{2\alpha}^e}\,\cdot(-\partial_\alpha \boldsymbol{y}\cdot\boldsymbol{d}_1^\theta)
       \vspace{3pt}\\
       \qquad+
       \dfrac{\partial\tilde W}{\partial \,K_{1\alpha}^e}\,\cdot(-\partial_\alpha \boldsymbol{d}_1^\theta\cdot\boldsymbol{d}_3)+ \dfrac{\partial\tilde W}{\partial \,K_{2\alpha}^e}\,\cdot(-\partial_\alpha \boldsymbol{d}_2^\theta\cdot\boldsymbol{d}_3)=0.
    \end{array}
\end{equation}
Inserting \eqref{14} into \eqref{20} we get the equivalent form
\begin{equation}\label{21}
    \begin{array}{l}
      \sin\theta\,\Big[ \dfrac{\partial\tilde W}{\partial \,E_{1\alpha}^e}\,\cdot(\partial_\alpha \boldsymbol{y}\cdot\boldsymbol{d}_1)+  \dfrac{\partial\tilde W}{\partial \,E_{2\alpha}^e}\,\cdot(\partial_\alpha \boldsymbol{y}\cdot\boldsymbol{d}_2)
       +
       \dfrac{\partial\tilde W}{\partial \,K_{1\alpha}^e}\,\cdot(\partial_\alpha \boldsymbol{d}_2\cdot\boldsymbol{d}_3)
        \vspace{3pt}\\+ \dfrac{\partial\tilde W}{\partial \,K_{2\alpha}^e}\,\cdot(-\partial_\alpha \boldsymbol{d}_1\cdot\boldsymbol{d}_3)\Big]
       +
              \cos\theta\, \Big[\dfrac{\partial\tilde W}{\partial \,E_{1\alpha}^e}\,\cdot(-\partial_\alpha \boldsymbol{y}\cdot\boldsymbol{d}_2)+  \dfrac{\partial\tilde W}{\partial \,E_{2\alpha}^e}\,\cdot(\partial_\alpha \boldsymbol{y}\cdot\boldsymbol{d}_1)
        \vspace{3pt}\\+
       \dfrac{\partial\tilde W}{\partial \,K_{1\alpha}^e}\,\cdot(\partial_\alpha \boldsymbol{d}_1\cdot\boldsymbol{d}_3)+ \dfrac{\partial\tilde W}{\partial \,K_{2\alpha}^e}\,\cdot(\partial_\alpha \boldsymbol{d}_2\cdot\boldsymbol{d}_3) \Big]=0,
    \end{array}
\end{equation}
which must hold for every angle  $\theta=\theta(x_1,x_2)$. Thus, both square brackets in \eqref{21} must be zero. We show next that the relation (21) implies the equation
\begin{equation}\label{22}
    \dfrac{\partial W}{\partial \boldsymbol{E}^e}\,\cdot(\boldsymbol{c}\,\boldsymbol{Q}^{e,T}\boldsymbol{F})+ \dfrac{\partial W}{\partial \boldsymbol{K}^e}\,\cdot(\boldsymbol{c}\boldsymbol{K})=0,
\end{equation}
where $\,\cdot\,$ denotes the scalar product of tensors and $\boldsymbol{c}$ is defined in \eqref{0,5}. Indeed, we have
\begin{equation*}
    \begin{array}{l}
      \dfrac{\partial W}{\partial \boldsymbol{E}^e}\,= \dfrac{\partial\tilde W}{\partial \,E_{i\alpha}^e}\,\,\boldsymbol{d}_i^0\otimes \boldsymbol{a}_\alpha\,,\qquad \dfrac{\partial W}{\partial \boldsymbol{K}^e}\,= \dfrac{\partial\tilde W}{\partial \,K_{i\alpha}^e}\,\,\boldsymbol{d}_i^0\otimes \boldsymbol{a}_\alpha\,,
      \vspace{4pt}\\
      \boldsymbol{c}\,\boldsymbol{Q}^{e,T}\boldsymbol{F}= (\epsilon_{\alpha\beta}\boldsymbol{d}_\alpha^0\otimes \boldsymbol{d}_\beta^0 ) (\boldsymbol{d}_i^0\otimes\boldsymbol{d}_i) (\partial_\gamma \boldsymbol{y}\otimes\boldsymbol{a}^\gamma) = \epsilon_{\alpha\beta} (\boldsymbol{d}_\alpha^0\otimes\boldsymbol{d}_\beta) (\partial_\gamma \boldsymbol{y}\otimes\boldsymbol{a}^\gamma)
      \vspace{3pt}\\
      \qquad\qquad =
       \epsilon_{\alpha\beta} (\partial_\gamma \boldsymbol{y}\cdot\boldsymbol{d}_\beta)  \boldsymbol{d}_\alpha^0\otimes\boldsymbol{a}^\gamma=
       (\partial_\alpha \boldsymbol{y}\cdot\boldsymbol{d}_2)  \boldsymbol{d}_1^0\otimes\boldsymbol{a}^\alpha -(\partial_\alpha \boldsymbol{y}\cdot\boldsymbol{d}_1)  \boldsymbol{d}_2^0\otimes\boldsymbol{a}^\alpha,
      \vspace{4pt}\\
      \boldsymbol{c}\boldsymbol{K}= (\epsilon_{\alpha\beta}\boldsymbol{d}_\alpha^0\otimes \boldsymbol{d}_\beta^0 ) [(\partial_\gamma \boldsymbol{d}_2\cdot\boldsymbol{d}_3)\boldsymbol{d}_1^0
    +(\partial_\gamma \boldsymbol{d}_3\cdot\boldsymbol{d}_1)\boldsymbol{d}_2^0
    +(\partial_\gamma \boldsymbol{d}_1\cdot\boldsymbol{d}_2)\boldsymbol{d}_3^0\,] \otimes\boldsymbol{a}^\gamma
    \vspace{3pt}\\
    \qquad = -(\partial_\alpha \boldsymbol{d}_1\cdot\boldsymbol{d}_3) \boldsymbol{d}_1^0\otimes\boldsymbol{a}^\alpha -
    (\partial_\alpha \boldsymbol{d}_2\cdot\boldsymbol{d}_3) \boldsymbol{d}_2^0\otimes\boldsymbol{a}^\alpha
    \end{array}
\end{equation*}
and inserting this in the left-hand side of \eqref{22} we obtain the last square bracket in \eqref{21}, which is equal to zero.

Thus, the relation \eqref{22} holds true. In view of \eqref{1} and \eqref{4} we can express the relation \eqref{22} in terms of $\boldsymbol{E}^e$ and $\boldsymbol{K}^e$ in the form
\begin{equation}\label{23}
    \dfrac{\partial W}{\partial \boldsymbol{E}^e}\,\cdot\boldsymbol{c}(\boldsymbol{E}^e+\boldsymbol{a})+ \dfrac{\partial W}{\partial \boldsymbol{K}^e}\,\cdot\boldsymbol{c}(\boldsymbol{K}^e+\boldsymbol{K}^0)=0.
\end{equation}

We regard the relation \eqref{23} as a first order linear partial differential equation for the unknown function $W(\boldsymbol{E}^e,\boldsymbol{K}^e)$. The characteristic system attached to the differential equation \eqref{23} is (see e.g. \cite{Vrabie-2003}, Sect. 6.3)
\begin{equation}\label{24}
    \dfrac{\mathrm{d}\boldsymbol{E}^e}{\mathrm{d}s} = \boldsymbol{c}(\boldsymbol{E}^e+\boldsymbol{a}), \qquad \dfrac{\mathrm{d}\boldsymbol{K}^e}{\mathrm{d}s} = \boldsymbol{c}(\boldsymbol{K}^e+\boldsymbol{K}^0).
\end{equation}
Since the unknown function $W(\boldsymbol{E}^e,\boldsymbol{K}^e)$ depends in total on 12 independent scalar arguments (6 components of $\boldsymbol{E}^e$ and 6 components of $\boldsymbol{K}^e$ in the tensor basis $\{\boldsymbol{d}_i^0\otimes \boldsymbol{a}^\alpha\}$) it suffices to determine 11 independent first integrals of the system of ordinary differential equations \eqref{24}. Then, the general solution of the equation \eqref{22} can be represented as an arbitrary function of the 11 first integrals.

Let us introduce the functions $\boldsymbol{U}_1$ , $\boldsymbol{U}_2$ , $\boldsymbol{U}_3$ , and $\boldsymbol{U}_4$ given by
\begin{equation}\label{25}
    \begin{array}{l}
      \boldsymbol{U}_1=(\boldsymbol{E}^e+\boldsymbol{a})^T (\boldsymbol{E}^e+\boldsymbol{a}),\qquad
   \boldsymbol{U}_2=\boldsymbol{n}^0\boldsymbol{E}^e, \vspace{3pt}\\
   \boldsymbol{U}_3=(\boldsymbol{E}^e+\boldsymbol{a})^T\boldsymbol{c}\, (\boldsymbol{K}^e+\boldsymbol{K}^0),\qquad
    \boldsymbol{U}_4=\boldsymbol{n}^0\boldsymbol{K}^e.
    \end{array}
\end{equation}
We show that $\boldsymbol{U}_1$ , $\boldsymbol{U}_2$ , $\boldsymbol{U}_3$ ,  $\boldsymbol{U}_4$ represent first integrals of the system \eqref{24}. Indeed, taking into account $\boldsymbol{n}^0\boldsymbol{c}=\boldsymbol{0}\,$, $\,\boldsymbol{a}^T=\boldsymbol{a}$ ,  and $\,\,\boldsymbol{c}^T=-\boldsymbol{c}$ and \eqref{24} we find
\begin{equation*}
    \begin{array}{l}
       \dfrac{\mathrm{d}\boldsymbol{U}_1}{\mathrm{d}s} = \Big(\dfrac{\mathrm{d}\boldsymbol{E}^e}{\mathrm{d}s}\Big)^T
       (\boldsymbol{E}^e+\boldsymbol{a}) + (\boldsymbol{E}^{e,T}+\boldsymbol{a})\dfrac{\mathrm{d}\boldsymbol{E}^e}{\mathrm{d}s}
       \vspace{3pt}\\
       \qquad\quad =
       (\boldsymbol{E}^{e,T}+\boldsymbol{a})(-\boldsymbol{c})(\boldsymbol{E}^e+\boldsymbol{a})
       +(\boldsymbol{E}^{e,T}+\boldsymbol{a})\boldsymbol{c}\,(\boldsymbol{E}^e+\boldsymbol{a})
       =   \boldsymbol{0},
       \vspace{4pt}\\
    \dfrac{\mathrm{d}\boldsymbol{U}_2}{\mathrm{d}s} = \boldsymbol{n}^0 \,\, \dfrac{\mathrm{d}\boldsymbol{E}^e}{\mathrm{d}s}=
        \boldsymbol{n}^0 [\boldsymbol{c}(\boldsymbol{E}^e+\boldsymbol{a})] = \boldsymbol{0},
    \vspace{4pt}\\
    \dfrac{\mathrm{d}\boldsymbol{U}_3}{\mathrm{d}s} = \Big(\dfrac{\mathrm{d}\boldsymbol{E}^e}{\mathrm{d}s}\Big)^T
       \boldsymbol{c}(\boldsymbol{K}^e+\boldsymbol{K}^0)+ (\boldsymbol{E}^{e,T}+\boldsymbol{a})\boldsymbol{c}\, \dfrac{\mathrm{d}\boldsymbol{K}^e}{\mathrm{d}s}
       \vspace{3pt}\\
       \qquad\quad =
       (\boldsymbol{E}^{e,T}+\boldsymbol{a})(-\boldsymbol{c})\boldsymbol{c} \, (\boldsymbol{K}^e+\boldsymbol{K}^0)
       +(\boldsymbol{E}^{e,T}+\boldsymbol{a})\boldsymbol{c}\,\boldsymbol{c}\,
       (\boldsymbol{K}^e+\boldsymbol{K}^0)   =   \boldsymbol{0},
       \vspace{4pt}\\
       \dfrac{\mathrm{d}\boldsymbol{U}_4}{\mathrm{d}s} = \boldsymbol{n}^0 \,\, \dfrac{\mathrm{d}\boldsymbol{K}^e}{\mathrm{d}s}=    \boldsymbol{n}^0 [\boldsymbol{c}(\boldsymbol{K}^e+\boldsymbol{K}^0)] = \boldsymbol{0}.
     \end{array}
\end{equation*}
The functions $\boldsymbol{U}_1$ , $\boldsymbol{U}_2$ , $\boldsymbol{U}_3$ ,  $\boldsymbol{U}_4$ give in total 11 independent first integrals: 3 components of $\boldsymbol{U}_1$  (symmetric) and 4 components of $\boldsymbol{U}_3$ in the tensor basis $\{\boldsymbol{d}_\alpha^0\otimes \boldsymbol{a}^\beta\}$, 2 components of the vector   $\boldsymbol{U}_2$ and 2 components of the vector $\boldsymbol{U}_4$ in the   basis $\{ \boldsymbol{a}^1,\boldsymbol{a}^2 \}$.

According to the  theory of differential equations (see e.g. \cite{Vrabie-2003}, Sect. 6.1), the general solution of the partial differential equation \eqref{23} has the form
\begin{equation}\label{26}
    W(\boldsymbol{E}^e,\boldsymbol{K}^e)= \hat W(\boldsymbol{U}_1\,,\,\boldsymbol{U}_2\,,\,\boldsymbol{U}_3\,,\,\boldsymbol{U}_4).
\end{equation}
On the other hand, by virtue of \eqref{19} we have $\,\,\dfrac{\partial W}{\partial \boldsymbol{U}_4}\,=\boldsymbol{0}\,$ and from \eqref{26} we deduce that the energy function $W$ can be represented as
\begin{equation}\label{27}
     W(\boldsymbol{E}^e,\boldsymbol{K}^e)= \hat W(\boldsymbol{U}_1\,,\,\boldsymbol{U}_2\,,\,\boldsymbol{U}_3 ).
\end{equation}
Finally, by a straightforward calculation we obtain from \eqref{25} and \eqref{1}--\eqref{5} the relations
\begin{equation*}
    \begin{array}{l}
       \boldsymbol{U}_1=(\boldsymbol{E}^e+\boldsymbol{a})^T (\boldsymbol{E}^e+\boldsymbol{a})
       =(\boldsymbol{Q}^{e,T}\boldsymbol{F})^T(\boldsymbol{Q}^{e,T}\boldsymbol{F})= \boldsymbol{F}^T\boldsymbol{F}, \vspace{3pt}\\
      \boldsymbol{U}_2=\boldsymbol{n}^0\boldsymbol{E}^e= \boldsymbol{d}_3^0(\boldsymbol{Q}^{e,T}\partial_\alpha\boldsymbol{y}- \boldsymbol{a}_\alpha)\otimes\boldsymbol{a}^\alpha
      = \big[(\boldsymbol{Q}^e \boldsymbol{d}_3^0)\cdot\partial_\alpha\boldsymbol{y}\big] \boldsymbol{a}^\alpha
       = (\boldsymbol{d}_3 \cdot\partial_\alpha\boldsymbol{y}) \boldsymbol{a}^\alpha
       \vspace{3pt}\\
      \qquad
      = \boldsymbol{d}_3 (\partial_\alpha\boldsymbol{y}\otimes \boldsymbol{a}^\alpha)=
      \boldsymbol{F}^T\boldsymbol{d}_3\,,
      \vspace{3pt}\\
   \boldsymbol{U}_3=(\boldsymbol{E}^e+\boldsymbol{a})^T\boldsymbol{c}\, (\boldsymbol{K}^e+\boldsymbol{K}^0)= (\boldsymbol{Q}^{e,T}\boldsymbol{F})^T\boldsymbol{c}\, \boldsymbol{K} = \boldsymbol{F}^T\boldsymbol{Q}^{e}(\boldsymbol{d}_1^0\otimes \boldsymbol{d}_2^0-  \boldsymbol{d}_2^0\otimes \boldsymbol{d}_1^0)\boldsymbol{K}
   \vspace{3pt}\\
   \qquad = \boldsymbol{F}^T (\boldsymbol{d}_1\otimes \boldsymbol{d}_2^0-  \boldsymbol{d}_2\otimes \boldsymbol{d}_1^0) \big[(\partial_\alpha \boldsymbol{d}_2\cdot\boldsymbol{d}_3)\boldsymbol{d}_1^0
    +(\partial_\alpha \boldsymbol{d}_3\cdot\boldsymbol{d}_1)\boldsymbol{d}_2^0
    +(\partial_\alpha \boldsymbol{d}_1\cdot\boldsymbol{d}_2)\boldsymbol{d}_3^0\big] \otimes\boldsymbol{a}^\alpha
     \vspace{3pt}\\
   \qquad = \boldsymbol{F}^T \big[ (\partial_\alpha \boldsymbol{d}_3\cdot\boldsymbol{d}_1)\boldsymbol{d}_1
    +(\partial_\alpha \boldsymbol{d}_3\cdot\boldsymbol{d}_2)\boldsymbol{d}_2\big]  \otimes\boldsymbol{a}^\alpha =  \boldsymbol{F}^T \big[ (\boldsymbol{d}_1\otimes \boldsymbol{d}_1+  \boldsymbol{d}_2\otimes \boldsymbol{d}_2)\,\partial_\alpha \boldsymbol{d}_3\big]  \otimes\boldsymbol{a}^\alpha
    \vspace{3pt}\\
   \qquad = \boldsymbol{F}^T   (\boldsymbol{d}_i\otimes \boldsymbol{d}_i)(\partial_\alpha \boldsymbol{d}_3 \otimes\boldsymbol{a}^\alpha) =  \boldsymbol{F}^T  \,\mathrm{Grad}_s\boldsymbol{d}_3\,.
    \end{array}
\end{equation*}
Inserting these expressions for $\boldsymbol{U}_i$ in \eqref{27} we obtain that the representation \eqref{17} holds true. The proof is complete.
\end{proof}

\begin{remark}\label{rem1}
The values of the arguments of the function $\hat W$ in \eqref{17}, calculated in the reference configuration $S^0$, are
\begin{equation*}
    \begin{array}{c}
     \boldsymbol{F}^{ T}\boldsymbol{F}_{\big| S^0}= \boldsymbol{a}^T\boldsymbol{a}=\boldsymbol{a} , \qquad
        {\boldsymbol{F}^T \boldsymbol{d}_3}_{\big| S^0}= \boldsymbol{a}^T\boldsymbol{d}_3^0= \boldsymbol{0}\,,
         \vspace{4pt}\\
          {\boldsymbol{F}^T\mathrm{Grad}_s\boldsymbol{d}_3}_{\big| S^0}= \boldsymbol{a}^T\mathrm{Grad}_s\boldsymbol{d}_3^0= \mathrm{Grad}_s\boldsymbol{n}^0 = -\boldsymbol{b}\,.
\end{array}
\end{equation*}
Then, one can introduce the measures of deformation $\boldsymbol{\mathcal{E}}$ , $\boldsymbol{\gamma}$ , $\boldsymbol{\Psi}$ defined by
\begin{equation}\label{28}
    \begin{array}{l}
    \boldsymbol{\mathcal{E}}=\dfrac{1}{2}\,(\boldsymbol{F}^{ T}\boldsymbol{F} -\boldsymbol{a})\,\,,\qquad
    \boldsymbol{\gamma}=  \boldsymbol{F}^T\boldsymbol{d}_3\,, \vspace{4pt}\\
    \boldsymbol{\Psi}= (\boldsymbol{F}^T\mathrm{Grad}_s\boldsymbol{d}_3 + \boldsymbol{b} ) +\boldsymbol{\mathcal{E}} \boldsymbol{b}\, ,
\end{array}
\end{equation}
and the strain energy function \eqref{17} can be represented as
\begin{equation}\label{29}
    W=\check{W}(\boldsymbol{\mathcal{E}},\boldsymbol{\gamma},\boldsymbol{\Psi}).
\end{equation}
The tensor $\boldsymbol{\mathcal{E}}$ is a second order symmetric tensor accounting for extensional and in-plane shear strains, $\boldsymbol{\gamma}$ is the vector of transverse shear deformation, and $\boldsymbol{\Psi}$ is a second order tensor for the bending and twist strains.
We designate by $\,\boldsymbol{E}^e_\parallel\,=\boldsymbol{a}\boldsymbol{E}^e= E_{\beta\alpha}\boldsymbol{d}_\beta^0\otimes\boldsymbol{a}^\alpha$ the ``planar part'' of $\,\boldsymbol{E}^e$ (in the tangent plane) and by $\,\boldsymbol{E}^e_{\perp}=\boldsymbol{n^0}\boldsymbol{E}^e= E_{3\alpha} \boldsymbol{a}^\alpha$ the ``normal part'' of $\,\boldsymbol{E}^e$, and analogously $\,\boldsymbol{K}^e_\parallel\,=\boldsymbol{a}\boldsymbol{K}^e= K_{\beta\alpha}\boldsymbol{d}_\beta^0\otimes\boldsymbol{a}^\alpha$. Then, in view of \eqref{1}, \eqref{2}, the tensors \eqref{28} can be written in the alternative forms
\begin{equation}\label{30}
    \begin{array}{l}
   \boldsymbol{\mathcal{E}}= \frac{1}{2}\,\boldsymbol{E}^{e, T}\boldsymbol{E}^e+\mathrm{sym}( \boldsymbol{E}^e_\parallel)\,,\qquad
   \boldsymbol{\gamma}= \boldsymbol{n}^0\boldsymbol{E}^e=\boldsymbol{E}^e_{\perp}\, ,\vspace{4pt} \\
   \boldsymbol{\Psi}=  (\boldsymbol{E}^{e, T}+\boldsymbol{a})\boldsymbol{c} \boldsymbol{K}^e+ [\, \frac{1}{2}\,\boldsymbol{E}^{e, T}\boldsymbol{E}^e+ \mathrm{skew} (\boldsymbol{E}^e_\parallel)\,]\,\boldsymbol{b}\,,
\end{array}
\end{equation}
\end{remark}

\begin{remark}\label{rem2}
Instead of the first integral $\boldsymbol{U}_3$ introduced in the proof of the Theorem, one can consider alternatively the following first integral:
\begin{equation}\label{31}
    \boldsymbol{U}_5=(\boldsymbol{E}^{e,T}_\parallel+\boldsymbol{a}) (\boldsymbol{K}^e+\boldsymbol{K}^0).
\end{equation}
Indeed $ \boldsymbol{U}_5$ is a first integral of the system \eqref{24} since we have
\begin{equation*}
    \begin{array}{c}
       \dfrac{\mathrm{d}\boldsymbol{U}_5}{\mathrm{d}s} =
       \dfrac{\mathrm{d}}{\mathrm{d}s}\big[(\boldsymbol{E}^{e,T} +\boldsymbol{a}) \boldsymbol{a} (\boldsymbol{K}^e+\boldsymbol{K}^0)\big]=
       \Big(\dfrac{\mathrm{d}\boldsymbol{E}^e}{\mathrm{d}s}\Big)^T
       \boldsymbol{a}(\boldsymbol{K}^e+\boldsymbol{K}^0)+ (\boldsymbol{E}^{e,T}+\boldsymbol{a})\boldsymbol{a}\, \dfrac{\mathrm{d}\boldsymbol{K}^e}{\mathrm{d}s}
       \vspace{4pt}\\
       \qquad\quad =
       (\boldsymbol{E}^{e,T}+\boldsymbol{a})(-\boldsymbol{c})\boldsymbol{a} \, (\boldsymbol{K}^e+\boldsymbol{K}^0)
       +(\boldsymbol{E}^{e,T}+\boldsymbol{a})\boldsymbol{a}\,\boldsymbol{c}\,
       (\boldsymbol{K}^e+\boldsymbol{K}^0)   =   \boldsymbol{0}.
     \end{array}
\end{equation*}
On the other hand, the function $ \boldsymbol{U}_5$ given by \eqref{31} can be expressed in terms of $\boldsymbol{F}$ and $\mathrm{Grad}_s\boldsymbol{d}_3$ as follows
\begin{equation*}
    \begin{array}{l}
      \boldsymbol{U}_5=(\boldsymbol{E}^e+\boldsymbol{a})^T\boldsymbol{a}\, (\boldsymbol{K}^e+\boldsymbol{K}^0)= \boldsymbol{F}^T\boldsymbol{Q}^{e}\boldsymbol{a}\, \boldsymbol{K}
      \vspace{3pt}\\
   \qquad
   = \boldsymbol{F}^T\boldsymbol{Q}^{e}\boldsymbol{a}
      \big[(\partial_\alpha \boldsymbol{d}_2\cdot\boldsymbol{d}_3)\boldsymbol{d}_1^0
    +(\partial_\alpha \boldsymbol{d}_3\cdot\boldsymbol{d}_1)\boldsymbol{d}_2^0
    +(\partial_\alpha \boldsymbol{d}_1\cdot\boldsymbol{d}_2)\boldsymbol{d}_3^0\big] \otimes\boldsymbol{a}^\alpha
     \vspace{3pt}\\
   \qquad = \boldsymbol{F}^T(\boldsymbol{d}_i\otimes \boldsymbol{d}_i)\big[(\partial_\alpha \boldsymbol{d}_2\cdot\boldsymbol{d}_3)\boldsymbol{d}_1^0
    +(\partial_\alpha \boldsymbol{d}_3\cdot\boldsymbol{d}_1)\boldsymbol{d}_2^0
    \big]   \otimes\boldsymbol{a}^\alpha
    \vspace{3pt}\\
   \qquad =
        \boldsymbol{F}^T \big[(-\partial_\alpha \boldsymbol{d}_3\cdot\boldsymbol{d}_2)\boldsymbol{d}_1
    +(\partial_\alpha \boldsymbol{d}_3\cdot\boldsymbol{d}_1)\boldsymbol{d}_2
    \big]   \otimes\boldsymbol{a}^\alpha
       = \boldsymbol{F}^T   (\boldsymbol{d}_2\otimes \boldsymbol{d}_1 - \boldsymbol{d}_1\otimes \boldsymbol{d}_2)(\partial_\alpha \boldsymbol{d}_3 \otimes\boldsymbol{a}^\alpha)
        \vspace{3pt}\\
   \qquad
    =
    \boldsymbol{F}^T \big[(\boldsymbol{d}_3\times\boldsymbol{d}_i)\otimes \boldsymbol{d}_i\big] \,\mathrm{Grad}_s\boldsymbol{d}_3
    =
    \boldsymbol{F}^T \big(\boldsymbol{d}_3\times 1\!\!1_3\big) \,\mathrm{Grad}_s\boldsymbol{d}_3
    =
    \boldsymbol{F}^T \big(\boldsymbol{d}_3\times \,\mathrm{Grad}_s\boldsymbol{d}_3 \big)\,.
    \end{array}
\end{equation*}
Thus, if we employ the first integrals $ \boldsymbol{U}_1$ , $ \boldsymbol{U}_2$, and $ \boldsymbol{U}_5$ we obtain the following alternative representation of the strain energy function
\begin{equation}\label{32}
    W =
      \bar W\big(\boldsymbol{F}^{ T}\boldsymbol{F} \,\,, \,\,
        \boldsymbol{F}^T \boldsymbol{d}_3 \,\,, \,\,  \boldsymbol{F}^T (\boldsymbol{d}_3\times \,\mathrm{Grad}_s\boldsymbol{d}_3  )\, \big).
\end{equation}
Like in Remark \ref{rem1}, one can introduce the measures of deformation  $\boldsymbol{\mathcal{E}}$ , $\boldsymbol{\gamma}$ and $\boldsymbol{\Phi}$ defined by \eqref{28}$_{1,2}$ and respectively
\begin{equation}\label{33}
    \begin{array}{l}
   \boldsymbol{\Phi}
    = [\boldsymbol{F}^T(\boldsymbol{d}_3\times\mathrm{Grad}_s\boldsymbol{d}_3 ) + \boldsymbol{n}^0 \times\boldsymbol{b}\,]+\boldsymbol{\mathcal{E}}(\boldsymbol{n}^0 \times  \boldsymbol{b}).
\end{array}
\end{equation}
The tensor $\boldsymbol{\Phi}$ accounts for bending and twist strains and was previously introduced by Zhilin \cite{Zhilin06}. It can be rewritten  in the form
\begin{equation}\label{34}
    \begin{array}{l}
   \boldsymbol{\Phi}
    = (\boldsymbol{E}^{e, T}+\boldsymbol{a})\boldsymbol{K}^e_\parallel- [ \frac{1}{2}\,\boldsymbol{E}^{e, T}\boldsymbol{E}^e+ \mathrm{skew} (\boldsymbol{E}^e_\parallel)\,]\,\boldsymbol{c}\,\boldsymbol{b}.
\end{array}
\end{equation}
With the help of $\,\boldsymbol{\mathcal{E}}$ , $\boldsymbol{\gamma}$ and $\boldsymbol{\Phi}$  one can give the following representation for the strain energy function
\begin{equation}\label{35}
    W=\breve{\breve{W}}(\boldsymbol{\mathcal{E}},\boldsymbol{\gamma},\boldsymbol{\Phi}).
\end{equation}
The form \eqref{35} was established previously in \cite{Zhilin06}. The only difference to the representation \eqref{29} is the definition of the bending--twist tensor $\boldsymbol{\Phi}$ in \eqref{33}, as compared to $\boldsymbol{\Psi}$ in \eqref{28}$_3\,$.

By comparison of the relations  \eqref{28}$_3\,$ and \eqref{33} one can see that the definition \eqref{28}$_3\,$ for the bending--twist tensor is more appropriate since the vector product in \eqref{33} introduces an additional (unnecessary) rotation of $\,\,\mathrm{Grad}_s\boldsymbol{d}_3\,$ in the plane of $\,\{\boldsymbol{d}_1,\boldsymbol{d}_2\}$.
\end{remark}

\begin{remark}\label{rem3}
Let us write the above representations in the linearized theory. For the linear theory we introduce the small (infinitesimal) displacement $\,\boldsymbol{u}=\boldsymbol{y}-\boldsymbol{y}^0$ and the vector of small rotations $\boldsymbol{\psi}$ such that \cite{Pietraszkiewicz-book04,Zhilin06}
$$ \boldsymbol{Q}^e\stackrel{.}{=}\, 1\!\!1_3+ \boldsymbol{\psi}\times 1\!\!1_3\,,\qquad  \boldsymbol{Q}^{e,T}\stackrel{.}{=}\, 1\!\!1_3- \boldsymbol{\psi}\times 1\!\!1_3\,.$$
Then, we have
\begin{equation*}
    \begin{array}{c}
       \boldsymbol{F}\stackrel{.}{=}\,\mathrm{Grad}_s\boldsymbol{u}+\boldsymbol{a}= (\partial_\alpha\boldsymbol{u}+\boldsymbol{a}_\alpha)\otimes\boldsymbol{a}^\alpha,
       \vspace{4pt}\\
       \mathrm{axl}(\partial_\alpha \boldsymbol{Q}^e  \boldsymbol{Q}^{e,T})= \partial_\alpha\boldsymbol{\psi},\qquad \mathrm{axl}(\dot{ \boldsymbol{Q}^e  } \boldsymbol{Q}^{e,T})= \dot{\boldsymbol{\psi}}
    \end{array}
\end{equation*}
and from \eqref{1}, \eqref{2} we find, in the approximation of the linear theory,
\begin{equation}\label{36}
\begin{array}{l}
    \boldsymbol{E}^e\stackrel{.}{=}\, \mathrm{Grad}_s\boldsymbol{u}-(\boldsymbol{\psi}\times \boldsymbol{a})= (\partial_\alpha\boldsymbol{u}- \boldsymbol{\psi}\times \boldsymbol{a}_\alpha)\otimes\boldsymbol{a}^\alpha,\vspace{4pt}\\
    \boldsymbol{K}^e\stackrel{.}{=}\, \mathrm{Grad}_s\boldsymbol{\psi}= \partial_\alpha\boldsymbol{\psi}\otimes\boldsymbol{a}^\alpha.
    \end{array}
\end{equation}
Using \eqref{30}, \eqref{34} and \eqref{36} we find the expressions of  $\boldsymbol{\mathcal{E}}$ , $\boldsymbol{\gamma}$ and $\boldsymbol{\Psi}$ in the linear theory
\begin{equation}\label{37}
    \begin{array}{l}
    \boldsymbol{\mathcal{E}}\stackrel{.}{=} \mathrm{sym}(\boldsymbol{a}\,\mathrm{Grad}_s\boldsymbol{u}),\qquad
    \boldsymbol{\gamma}\stackrel{.}{=}    \boldsymbol{n}^0\mathrm{Grad}_s\boldsymbol{u}+ \boldsymbol{c}\,\boldsymbol{\psi}, \vspace{4pt}\\
    \boldsymbol{\Psi}\stackrel{.}{=}  \boldsymbol{c}\,\boldsymbol{\Phi}
    \stackrel{.}{=}   \boldsymbol{c}\,\mathrm{Grad}_s(\boldsymbol{a}\,\boldsymbol{\psi})+ [\mathrm{skew}(\boldsymbol{a}\,\mathrm{Grad}_s\boldsymbol{u})] \,\boldsymbol{b},
\end{array}
\end{equation}
We note that the relation between the tensors $\boldsymbol{\Psi}$  and $\boldsymbol{\Phi}$ in the linear theory is very simple: $\boldsymbol{\Psi}\stackrel{.}{=}\boldsymbol{c}\boldsymbol{\Phi}\,$.
If we decompose the vector of small rotations as $\boldsymbol{\psi}=\psi_i\boldsymbol{a}^i$, then the drilling rotations are described by the component $\boldsymbol{n}^0\cdot\boldsymbol{\psi}= \psi_3\,$. We remark that
\begin{equation*}
    \begin{array}{c}
      \boldsymbol{c}\,\boldsymbol{\psi}=\dfrac{1}{\sqrt{a}}\,(\boldsymbol{a}_1\otimes \boldsymbol{a}_2- \boldsymbol{a}_2\otimes \boldsymbol{a}_1)\boldsymbol{\psi} = \dfrac{1}{\sqrt{a}}\,(\psi_2\boldsymbol{a}_1 - \psi_1\boldsymbol{a}_2),
       \vspace{4pt}\\
       \mathrm{Grad}_s(\boldsymbol{a}\,\boldsymbol{\psi})= \mathrm{Grad}_s(\psi_\beta\boldsymbol{a}^\beta) = \partial_\alpha(\psi_\beta\boldsymbol{a}^\beta) \otimes\boldsymbol{a}^\alpha
    \end{array}
\end{equation*}
and from \eqref{37} we see that the tensors $\boldsymbol{\mathcal{E}}$ , $\boldsymbol{\gamma}$ and $\boldsymbol{\Psi}$ are indeed independent of the drilling rotations $\psi_3\,$.

In this case one gets the Reissner-type kinematics of shells \cite{Wisniewski10,Neff_Hong_Reissner08} with 5 degrees of freedom.
\end{remark}

\section{The case of isotropic shells}\label{sec4}

The local symmetry group for 6-parameter elastic shells has been studied in \cite{Eremeyev06}. The expression of the strain energy for a physically linear model has the general form
\begin{equation}\label{38}
    \begin{array}{crl}
   2 W(\boldsymbol{E}^e,\boldsymbol{K}^e)&=& \alpha_1\big(\mathrm{tr}  \boldsymbol{E}^e_{\parallel}\big)^2 +\alpha_2\,  \mathrm{tr} \big(\boldsymbol{E}^e_{\parallel}\big)^2    + \alpha_3\, \mathrm{tr}\big(\boldsymbol{E}^{e,T}_{\parallel}  \boldsymbol{E}^e_{\parallel} \big)  + \alpha_4     (\boldsymbol{n}^0 \boldsymbol{E}^e)^2 \vspace{3pt} \\
  &&    + \beta_1\big(\mathrm{tr}  \boldsymbol{K}^e_{\parallel}\big)^2  +\beta_2 \, \mathrm{tr} \big(\boldsymbol{K}^e_{\parallel}\big)^2    + \beta_3 \, \mathrm{tr}\big(\boldsymbol{K}^{e,T}_{\parallel} \boldsymbol{K}^e_{\parallel} \big)  + \beta_4     (\boldsymbol{n}^0 \boldsymbol{K}^e)^2\!  .
\end{array}
\end{equation}
The constitutive coefficients $\alpha_1$ ,..., $\alpha_4$ , $\beta_1$ ,..., $\beta_4$ can depend in general on the initial structure curvature tensor  $ \boldsymbol{K}^0$, but we  assume for simplicity that they are constant. Provided that the coefficients $\alpha_k$ and $\beta_k$  satisfy the following inequalities
\begin{equation}\label{39}
    \begin{array}{l}
    2\alpha_1+\alpha_2+\alpha_3>0,\qquad \alpha_2+\alpha_3>0,\qquad
    \alpha_3-\alpha_2>0, \qquad  \alpha_4>0,\\
    2\beta_1+\beta_2+\beta_3>0,\qquad
     \beta_2+\beta_3>0,\qquad \beta_3-\beta_2>0,\qquad  \beta_4>0,
\end{array}
\end{equation}
the energy function \eqref{38} is coercive in the sense that there exists a constant $C>0$ with
\begin{equation*}
    W( {\boldsymbol{E}^e}, {\boldsymbol{K}^e})\,\geq\, C\big( \,   \|\boldsymbol{E}^e\|^2 +  \|\boldsymbol{K}^e\|^2\,\big).
\end{equation*}
Under the conditions \eqref{39} we can prove the existence of minimizers for isotropic elastic shells. To this aim we apply the recent existence result in the theory of 6--parameter shells given by Theorem 1 in \cite{Birsan-Neff-MMS-2013}, see also \cite{Neff_plate04_cmt}.

\begin{remark}\label{rem4}
One can find in the literature some simplified versions of the strain energy \eqref{38} for 6-parameter isotropic shells. For instance, in \cite{Pietraszkiewicz-book04,Pietraszkiewicz10} the following special form is employed
\begin{equation}\label{40}
   \begin{array}{l}
    2W(\boldsymbol{E}^e,\boldsymbol{K}^e)= \,\,\,C\big[\,\nu \,(\mathrm{tr} \boldsymbol{E}^e_{\parallel})^2 +(1-\nu)\, \mathrm{tr}(\boldsymbol{E}^{e,T}_{\parallel}  \boldsymbol{E}^e_{\parallel} )\big]  + \alpha_{s\,}C(1-\nu) \, \boldsymbol{n}^0 \boldsymbol{E}^e \boldsymbol{E}^{e,T}  \boldsymbol{n}^0
    \vspace{3pt}\\
    \qquad\qquad\quad\,\, +\,D \big[\,\nu\,(\mathrm{tr} \boldsymbol{K}^e_{\parallel})^2 + (1-\nu)\, \mathrm{tr}(\boldsymbol{K}^{e,T}_{\parallel}  \boldsymbol{K}^e_{\parallel} )\big]  + \alpha_{t\,}D(1-\nu) \,    \boldsymbol{n}^0 \boldsymbol{K}^e \boldsymbol{K}^{e,T}  \boldsymbol{n}^0,
\end{array}
\end{equation}
where $C=\frac{E\,h}{1-\nu^2}\,$ is the stretching (in-plane) stiffness of the shell, $D=\frac{E\,h^3}{12(1-\nu^2)}\,$ is the bending stiffness, $h$ is the thickness of the shell, and $\alpha_s\,$, $\alpha_t$ are two shear correction factors. Also,    $E $ and $ \nu$ denote the Young modulus and  Poisson ratio of the isotropic and homogeneous material. By the numerical treatment of non-linear shell problems, the values of the shear correction factors have been set to  $\alpha_s=5/6$, $\alpha_t=7/10$ in \cite{Pietraszkiewicz10}.
In this simpler case, the coefficients $\alpha_k$ and $\beta_k$ from \eqref{38} have the expressions
\begin{equation}\label{41}
    \begin{array}{l}
    \alpha_1=C\nu,\qquad \alpha_2=0,\qquad \alpha_3=C(1-\nu),\qquad
     \alpha_4=\alpha_sC(1-\nu), \\     \beta_1=D\nu,\qquad \beta_2=0,\qquad
      \beta_3=D(1-\nu), \qquad \beta_4=\alpha_tD(1-\nu).
\end{array}
\end{equation}
We remark that the conditions \eqref{39} are satisfied for the values \eqref{41}, by virtue of the well-known inequalities $\,E>0\,$ and $\, -1<\nu<\dfrac{1}{2}\,\,$ (or equivalently, $\,\mu>0$ and $\, 2\mu+3\lambda>0$, in terms of the Lam\'e moduli $\lambda,\mu$).
\end{remark}

Let us consider next isotropic shells without drilling rotations. Zhilin \cite{Zhilin06} determined the following form of the strain energy $W$ as a quadratic function of its arguments $\boldsymbol{\mathcal{E}},\,\boldsymbol{\gamma},\,\boldsymbol{\Phi}$
\begin{equation}\label{42}
\begin{array}{r}
    2W=2\breve{\breve{W}}(\boldsymbol{\mathcal{E}},\boldsymbol{\gamma},\boldsymbol{\Phi})=
C\big[(1-\nu)\|\boldsymbol{\mathcal{E}}\|^2+\nu(\mathrm{tr}\,\boldsymbol{\mathcal{E}})^2\big] +\frac{1}{2}\,C(1-\nu)\kappa\,\boldsymbol{\gamma}^2
\vspace{4pt}\\
\qquad +
D\big[\,\|\boldsymbol{\Phi}\|^2 -\nu\,\mathrm{tr}( \boldsymbol{\Phi}^2)- \frac{1}{2}\,(1-\nu)\,(\mathrm{tr}\,\boldsymbol{\Phi})^2\,\big],
\end{array}
\end{equation}
where $\kappa$ is a shear correction factor. The role of shear correction factors has been extensively discussed in the literature, see e.g. \cite{Neff_Chelminski_ifb07}. For the determination of the constitutive coefficients presented in \eqref{42} Zhilin has employed the solutions of some shell problems within the linear theory. Taking into account the relations $\boldsymbol{\Phi}\stackrel{.}{=}-\boldsymbol{c}\boldsymbol{\Psi}\,$  and
$$ \mathrm{tr}\big[( \boldsymbol{c}\,\boldsymbol{\Psi})^2\,\big]= \mathrm{tr}\big(\boldsymbol{\Psi}^T\boldsymbol{\Psi}\big)- \big(\mathrm{tr}\boldsymbol{\Psi}\big)^2 = 2\|\,\mathrm{dev_2\,sym}\boldsymbol{\Psi}\,\|^2-
\mathrm{tr}( \boldsymbol{\Psi}^2),
$$
then we deduce from \eqref{42} the expression of the strain energy $W$ as function of $\boldsymbol{\mathcal{E}},\,\boldsymbol{\gamma},\,\boldsymbol{\Psi}\,$:
\begin{equation}\label{43}
    \begin{array}{r}
    2W=2 \check{W}(\boldsymbol{\mathcal{E}},\boldsymbol{\gamma},\boldsymbol{\Psi})=  C\big[\,(1-\nu)\|\boldsymbol{\mathcal{E}}\|^2+ \nu(\mathrm{tr}\,\boldsymbol{\mathcal{E}})^2\,\big]
      +\frac{1}{2}\,C(1-\nu)\kappa\,\boldsymbol{\gamma}^2 \vspace{4pt} \\
      + D\big[\,\frac{1}{2}\,(1-\nu)\|\boldsymbol{\Psi}\|^2
      +\frac{1}{2}\,(1-\nu)\,\mathrm{tr}( \boldsymbol{\Psi}^2)+\nu\,(\mathrm{tr}\,\boldsymbol{\Psi})^2\,\big],
    \end{array}
\end{equation}
or equivalently,
\begin{equation}\label{44}
\begin{array}{r}
    2W=2 \check{W}(\boldsymbol{\mathcal{E}},\boldsymbol{\gamma},\boldsymbol{\Psi})=  C\big[\,(1-\nu)\|\boldsymbol{\mathcal{E}}\|^2+ \nu(\mathrm{tr}\,\boldsymbol{\mathcal{E}})^2\,\big]
      +\frac{1}{2}\,C(1-\nu)\kappa\,\boldsymbol{\gamma}^2 \vspace{4pt} \\
      + D\big[\,\frac{1}{2}\,(1+\nu)(\mathrm{tr}\,\boldsymbol{\Psi})^2
      + (1-\nu)\,\|\,\mathrm{dev_2\,sym}\boldsymbol{\Psi}\,\|^2\,\big].
      \end{array}
\end{equation}
In order to compare this with the energy \eqref{38} for 6-parameter isotropic shells, we insert the expressions \eqref{30} into \eqref{44} and we find
\begin{equation}\label{45}
\begin{array}{l}
    2W( {\boldsymbol{E}^e}, {\boldsymbol{K}^e}) = C\big[\,\frac{1}{2}\,(1+\nu)\big(\frac{1}{2}\,\|\boldsymbol{E}^e\|^2+ \mathrm{tr}(\boldsymbol{E}^e_\parallel)\big)^2
    \vspace{4pt}\\
     \qquad   \quad\,\,\, + (1\!-\!\nu)\|\frac12\, \boldsymbol{E}^{e,T}\boldsymbol{E}^e-\frac14 \,\|\boldsymbol{E}^e\|^2 \boldsymbol{a} +
    \mathrm{dev_2\,sym}(\boldsymbol{E}^e_\parallel)\,\|^2\big] +
    \frac{1}{2}\,C (1\!-\!\nu)\,\kappa\,\|\boldsymbol{n}^0\boldsymbol{E}^e\|^2
    \vspace{4pt}\\
     \qquad   \quad\,\,\,
      +D\Big[\,\frac{1}{2}\,(1+\nu)\big(\mathrm{tr}
      (\boldsymbol{E}^{e,T}_\parallel\boldsymbol{c}\boldsymbol{K}^e_\parallel)+ \mathrm{tr}
      (\boldsymbol{c}\boldsymbol{K}^e_\parallel)+ \mathrm{tr}[ (\frac12\, \boldsymbol{E}^{e,T}\boldsymbol{E}^e+ \mathrm{skew}\,\boldsymbol{E}^e_\parallel)\,\boldsymbol{b}\,]\,\big)^2
      \vspace{4pt}\\
     \qquad \quad\,\,\,
      +(1-\nu)\big\| \, \mathrm{dev_2\,sym} \big[ \boldsymbol{E}^{e,T}_\parallel\boldsymbol{c}\boldsymbol{K}^e_\parallel + \boldsymbol{c}\boldsymbol{K}^e_\parallel+
        (\frac12\, \boldsymbol{E}^{e,T}\boldsymbol{E}^e+ \mathrm{skew}\,\boldsymbol{E}^e_\parallel\,)\,\boldsymbol{b}\,\big]\big\|^2\Big].
      \end{array}
\end{equation}
We observe that the energy \eqref{45} is super-quadratic as a function of the arguments $ (\boldsymbol{E}^e, \boldsymbol{K}^e)$. In the case of physically linear shells, when only the quadratic terms in $(\boldsymbol{E}^e, \boldsymbol{K}^e)$ are taken into account, we obtain the simplified expression of the energy density (for the case when the constitutive coefficients are independent of $\boldsymbol{K}^0$)
\begin{equation}\label{46}
    \begin{array}{l}
    2W( {\boldsymbol{E}^e}, {\boldsymbol{K}^e}) = C\big[\nu(\mathrm{tr}\,\boldsymbol{E}^e_\parallel)^2
    +\,\frac{1}{2}\, (1\!-\!\nu)\,   \mathrm{tr}(\boldsymbol{E}^e_\parallel)^2
    +\,\frac{1}{2}\, (1\!-\!\nu)\,  \mathrm{tr}(\boldsymbol{E}^{e,T}_\parallel\boldsymbol{E}^e_\parallel)\big]
    \vspace{4pt}\\
      \qquad\,\,\,
    +\frac{1}{2}\,C (1\!-\!\nu)\,\kappa\,\|\boldsymbol{n}^0\boldsymbol{E}^e\|^2
     +D\big[\,\mathrm{tr}(\boldsymbol{K}^{e,T}_\parallel\boldsymbol{K}^e_\parallel)- \frac{1}{2}\, (1\!-\!\nu)\,   (\mathrm{tr}\boldsymbol{K}^e_\parallel)^2 - \!\nu \,   \mathrm{tr}(\boldsymbol{K}^e_\parallel)^2\,\big].
    \end{array}
\end{equation}
By comparison of the relations \eqref{38} and \eqref{46} we see that the   coefficients $\alpha_k$ and $\beta_k$ that correspond to shells without drilling rotations are
\begin{equation}\label{47}
    \begin{array}{l}
    \alpha_1=C\,\nu,\qquad \alpha_2=\alpha_3=\frac{1}{2}\,C (1-\nu)\,,\qquad \alpha_4=\frac{1}{2}\,C (1-\nu)\,\kappa, \vspace{4pt}\\
    \beta_1=-\frac{1}{2}\,D (1-\nu)\,,\qquad \beta_2=-D\nu,\qquad \beta_3=D,\qquad \beta_4=0.
    \end{array}
\end{equation}
or equivalently,
\begin{equation*}
     \begin{array}{c}
    \alpha_1=h\,\dfrac{2\mu\lambda}{2\mu+\lambda}\,,\qquad \alpha_2=\alpha_3=h\,\mu\,,\qquad \alpha_4=h\,\mu\,\kappa, \vspace{4pt}\\
    \beta_1=-\dfrac{h^3}{12}\,\mu\,,\qquad \beta_2=-\dfrac{h^3}{12}\,\dfrac{2\mu\lambda}{2\mu+\lambda}\, \,,\qquad \beta_3=\dfrac{h^3}{12}\,\dfrac{4\mu(\mu+\lambda)}{2\mu+\lambda}\,,\qquad \beta_4=0.
    \end{array}
\end{equation*}

\begin{remark}\label{rem5}
We see that the inequalities \eqref{39} are not satisfied for the set of constitutive coefficients \eqref{47}. Indeed, we find
\begin{equation*}
 \begin{array}{c}
    2\alpha_1+\alpha_2+\alpha_3= h\,\dfrac{E}{1-\nu}\,=h\,\dfrac{2\mu(2\mu+3\lambda)}{2\mu+\lambda}\,
    >0,\qquad \alpha_2+\alpha_3= h\,\dfrac{E}{1+\nu}\,=2h\, \mu   >0, \vspace{4pt}\\
      \beta_2+\beta_3=\,\dfrac{h^3}{12}\,\dfrac{E}{1+\nu}\,=\,\dfrac{h^3}{6}\, \mu
      >0,\qquad \beta_3-\beta_2 =\, \dfrac{h^3}{12}\,\dfrac{E}{1-\nu}\,=\,\dfrac{h^3}{6}\, \dfrac{\mu(2\mu+3\lambda)}{2\mu+\lambda}\,
      >0.
       \end{array}
\end{equation*}
but also
\begin{equation}\label{50}
    \alpha_3-\alpha_2=0 , \qquad 2\beta_1+\beta_2+\beta_3=0, \qquad \beta_4=0.
\end{equation}
This means that the strain energy \eqref{46} for shells without drilling rotations is not positive definite, but only positive semi-definite. The existence theorem presented in \cite{Birsan-Neff-MMS-2013} does not apply here.
The proof of the existence of minimizers is more difficult in this case, but it can be pursued using the same methods as in Neff \cite{Neff_plate07_m3as}. In the works \cite{Neff_plate04_cmt,Neff_plate07_m3as} a plate model derived directly from the 3D equations of Cosserat elasticity is studied. The relation $\,\alpha_3-\alpha_2=0 \,$ from \eqref{50} corresponds to the case of zero Cosserat couple modulus ($\mu_c=0$) in \cite{Neff_plate04_cmt,Neff_plate07_m3as}. The  comparison between the 6-parameter resultant shell theory and the model developed in \cite{Neff_plate04_cmt,Neff_plate07_m3as} has been presented in \cite{Birsan-Neff-AnnRom-2012,Birsan-Neff-JElast-2013,Birsan-Neff-MMS-2013}.
\end{remark}

\begin{remark}\label{rem6}
In this section we have considered for simplicity the case when the constitutive coefficients $\,\,\alpha_k$ and $\beta_k\,$ are independent of the initial curvature tensor $\boldsymbol{K}^0$ (or equivalently on $\boldsymbol{b}$, since $\boldsymbol{a}\boldsymbol{K}^0=\boldsymbol{c}\boldsymbol{b}$). However, a similar analysis can be performed also in the more complicated case when  the constitutive coefficients of the strain energy function $W$  depend on  $\boldsymbol{K}^0$.
\end{remark}

\bibliographystyle{plain}
\bibliography{literatur_Birsan}

\end{document}